\documentclass[12pt,leqno,a4paper]{amsart}
\usepackage{amssymb}  %%%%%%%%% AMS Fonts & symbols
\usepackage{amsmath}  %%%%%%%%% Multline formulas .......
\usepackage{amsthm}  %%% For proof environment, proclamation styles ...
\usepackage{verbatim} %%% to have i.e. the comment environment
\theoremstyle{plain}
\newtheorem{theorem}{THEOREM}[section]
\newtheorem{proposition}{PROPOSITION}[section]
\newtheorem{lemma}{LEMMA}[section]

\theoremstyle{definition}
\newtheorem{definition}{DEFINITION}[section]

\newtheorem{remark}{Remark}[section]

\numberwithin{equation}{section}
\newcommand{\ga}{\alpha}
\newcommand{\gb}{\beta}
\newcommand{\gD}{\Delta}

\newcommand{\gth}{\vartheta}
\newcommand{\gTh}{\Theta}

\newcommand{\gl}{\lambda}
\newcommand{\gL}{\Lambda}
\newcommand{\gm}{\mu}
\newcommand{\gn}{\nu}

\newcommand{\gp}{\pi}

\newcommand{\gs}{\sigma}

\newcommand{\gf}{\varphi}

\newcommand{\gps}{\psi}

\newcommand{\gO}{\Omega}
\newcommand{\ov}{\overline}

\newcommand{\R}{\mathbb {R}}
\newcommand{\RN}{\R ^N}

\newcommand{\N}{\mathbb {N}}
\newcommand{\ldue}{L^2}
\newcommand{\huno}{H_0^1(\Omega \cup \Gamma _2)}

\newcommand{\noi}{\noindent}
\newcommand{\be}{\begin{equation}}
\newcommand{\ee}{\end{equation}}

\begin{document}

\title[Morse index and symmetry for mixed elliptic problems]
{Morse index and symmetry for elliptic problems with nonlinear mixed  boundary conditions}

\author[Damascelli]{Lucio Damascelli}
\address{ Dipartimento di Matematica, Universit\`a  di Roma
" Tor Vergata " - Via della Ricerca Scientifica 1 - 00173  Roma - Italy.}
\email{damascel@mat.uniroma2.it}
\author[Pacella]{Filomena Pacella}
\address{Dipartimento di Matematica, Universit\`a di Roma
" La Sapienza " -  P.le A. Moro 2 - 00185 Roma - Italy.}
\email{pacella@mat.uniroma1.it}
\date{}
\thanks{Supported by PRIN-2009-WRJ3W7 grant}
\begin{abstract}
Under a Morse index condition we prove symmetry results for solutions of a nonlinear mixed boundary condition elliptic problem. As an intermediate step we relate the Morse index of  a solution to a mixed boundary condition linear eigenvalue problem for which we construct sequences of eigenvalues and provide variational characterization of them.
\end{abstract}

\maketitle

\section{Introduction}
We consider an elliptic problems with mixed nonlinear boundary conditions of the type
\be  \label{genprob} 
\begin{cases} - \gD u =f(x,u) \quad &\text{in } \gO \\
u= 0 \quad & \text{on } \Gamma _1\\
\frac {\partial u} {\partial \gn}= g(x,u) \quad & \text{on }  \Gamma _2 
\end{cases}
\ee
 where  $\Omega $ is a bounded domain 
%  with Lipschitz boundary 
 in $\R^N$, $N \geq 2$, $\gn $ stands for the outer normal, and  $\Gamma _1$, $\Gamma _2 $ are relatively open nonempty disjoint subset of the boundary $\partial \Omega$ such that
 
\be
\label{GeneralBoundary1}   
\Gamma _2  \text{ is a smooth } (N-1)- \text{submanifold}   \quad , \quad \Gamma _1= \partial \Omega \setminus \overline {\Gamma _2} 
\ee
and  
\be \label{GeneralBoundary2}   
\partial \Omega \setminus (\Gamma _1 \cup \Gamma _2)= \overline {\Gamma _1} \cap \overline {\Gamma _2}  \text{ is a smooth }   (N-2)-\text{submanifold}
% \end{split}
\ee

\medskip
Moreover in all of the domains that we consider $\Gamma _1 $ is a smooth  $(N-1)$- submanifold, except possibly for a singular set $\Gamma ' \subset \Gamma _1$ which is discrete set or a smooth $(N-2)$-submanifold.
 \par
\bigskip
We will assume further that  $f=f(x,s):\overline { \Omega } \times \R  \to \R $,  $g=g(x,s): \overline {  \Gamma _2 } \times \R  \to \R $  are differentiable
 %  $C^1$ 
  with respect to $s$ and
 % the second variable, and 
\be \label{Regolarita-f,g}
f, \frac {\partial f} {\partial s}, g, \frac {\partial g} {\partial s} \text{ are locally H\" older continuous functions in } \overline {\Omega } \times \R 
 \ee
\medskip
A solution of \eqref{genprob} will be understood in a weak sense. \par
Therefore 
we denote by $H^{1}_0 (\gO \cup \Gamma _2) $ the closure of $C_c^{\infty}(\gO \cup \Gamma _2 )$ in the space $H^{1} (\Omega )$ (which coincides with
the space of functions $u \in H^1 (\Omega )$ such that
 the trace of $u$ vanishes on $\Gamma _1$), 
  and say that 
$u $ is a $C^1$ bounded weak solution of the problem, if $u \in  H^{1}_0 (\gO \cup \Gamma _2)  \cap C^1 ( \Omega)\cap L^{\infty }(\Omega)$
 and  
\be 
\int _{\gO } \nabla u \cdot \nabla \gf \, dx = \int _{\gO }f(x,u)  \gf \, dx + \int _{\Gamma} g(x',u)  \gf \, dx' 
\; \; \forall \, \gf \in H^{1}_0 (\gO \cup \Gamma _2)
\ee
\par
\smallskip

\begin{comment}

\end{comment}

The main aim of this paper is to prove cylindrical symmetry of solutions of \eqref{genprob}, both positive and sign changing, in some domains with cylindrical symmetry, by maximum principles and spectral properties of the linearized operator at the solution. \par
Denoting by $x=(x', x_N )$ a point $x=(x_1, \dots , x_{N-1}, x_N) \in \R^N$, the domains we consider will be subsets of the half space
$\R^N_{+}= \{ x=(x_1, \dots , x_N) \in \R^N:  x_N>0 \} 
$
defined in the following way.
\begin{definition}\label{DefDominiSimmetrici}
We say that a bounded domain $\Omega $ has  cylindrical symmetry if assuming that 
$$ \inf \{ t \in \R : (x',t) \in \Omega  \} =0 \; , \; 
    \sup \{ t \in \R : (x',t) \in \Omega  \} =b >0
$$
then
 for every $h \in (0,b)$ the set 
 $$\Omega ^h=\Omega  \cap \{ x_N=h \}
 $$
 is either a $N-1$-dimensional ball or a 
 $N-1$-dimensional annulus with the center on the $x_N$ axis, and
$$ \overline {\Omega ^0}= \overline {\Omega }  \cap \{ x_N=0 \}
 $$
 is also a nondegenerate closed ball or annulus in $\R^{N-1}$, whose nonempty interior in $\R^{N-1}$ we denote by 
 $\Omega ^0 $.\par
 \end{definition} 
\medskip
For such domains we will always assume that
\be \label{FrontieraDominiSimmetrici}
\Gamma _2= \Omega ^0 \quad ; \quad \Gamma _1= \partial \Omega \setminus \overline {\Gamma _2}
\ee
Thus $\Gamma _2 $ is a relatively open flat part of the boundary at the height $x_N=0$, which by our assumptions is either a $(N-1)$- dimensional ball or a 
$(N-1)$- dimensional annulus.\par
%  Moreover, up to $(N-2)$ smooth manifolds,  
%  \be \Gamma _1= \partial \Omega \setminus \overline {\Gamma _2}
%  \ee

\smallskip
\textbf{Examples} of such domains are \\
a half ball
 $$(B_R^N)_+= B_R \cap \R^N_{+}=  \{ x=(x_1, \dots , x_N) \in \R^N: |x| < R \, ; \, x_N>0 \} \quad ,
$$ 
a half annulus
$$(A_{R_1, R_2}^N)_+= \{ x=(x_1, \dots , x_N) \in \R^N: R_1<|x| < R_2 \, ; \, x_N>0 \} \quad ,
$$
a cylinder 
 $$C_{R,b}=  \{ x=(x' , x_N) \in \R^N: |x'| < R \, ; \, 0<x_N<b \} \quad ,
 $$
an annular cylinder
 $$C_{R_1, R_2,b}=  \{ x=(x' , x_N) \in \R^N: R_1<|x'| < R \, ; \, 0<x_N<b \} \quad ,
 $$
a cone
$$ K_{R,b} = \{ x=(x' , x_N) \in \R^N: |x'| < \frac Rb (b-x_N) \, ; \, 0<x_N<b \} \quad .
$$
\par
Note that $\Gamma _1 $ is smooth in these examples, with the exceptions of the  cone at the vertex, and the cylinders  at height $b$. \par
\medskip

\begin{comment}

$$(B_R^N)_+= B_R \cap \R^N_{+}=  \{ x=(x_1, \dots , x_N) \in \R^N: |x| < R \, ; \, x_N>0 \}
$$ 
with
$$\Gamma _1= (\partial \gO \cap \R^N_{+} ) =  \{(x', x_N):|x'| < R \, , \, x_N= \sqrt {R^2- |x'|^2} \}
$$
$$\Gamma _2 =( \partial \gO \cap \R^N_0 ) \setminus S_R^0= \{(x', x_N):|x'| < R \, , \, x_N=0  \}=B_R^{N-1}\times  \{ 0 \}
$$
  or 
a half annulus  
$$(A_{R_1, R_2}^N)_+= \{ x=(x_1, \dots , x_N) \in \R^N: R_1<|x| < R_2 \, ; \, x_N>0 \}
$$
with
$$\Gamma _1= (\partial \gO \cap \R^N_{+} ) =  \{(x', x_N):R_1 < |x'| < R_2 \, , \, x_N= \sqrt {R_1^2- |x'|^2} \text{ or } x_N= \sqrt {R_2^2- |x'|^2} \}
$$
$$\Gamma _2 =( \partial \gO \cap \R^N_0 ) \setminus (S_{R_1 }^0 \cup S_{R_1 }^0 )= \{(x', x_N):|x'| < R \, , \, x_N= 0 \}=
A_{R_1, R_2}^{N-1}\times  \{ 0 \}
$$
  or 
 a cylinder 
 $$C_{R,b}=  \{ x=(x' , x_N) \in \R^N: |x'| < R \, ; \, 0<x_N<b \}
 $$
 with
 $$\Gamma _1= \{  x=(x' , x_N) \in \partial \gO : 0<x_N <b \} 
$$
$$ \Gamma _2 = \{  x=(x' , x_N) \in \partial \gO : |x'| < R \, , \,  x_N =0  \text{ or } x_N=b \}
$$
  or 
  an annular cylinder 
 $$C_{R_1, R_2,b}=  \{ x=(x' , x_N) \in \R^N: R_1<|x'| < R \, ; \, 0<x_N<b \}
 $$
 with
$$\Gamma _1= \{  x=(x' , x_N) \in \partial \gO : 0<x_N <b \} 
$$
$$ \Gamma _2 = \{  x=(x' , x_N) \in \partial \gO : R_1<|x'| < R_2 \, , \,  x_N =0  \text{ or } x_N=b \}
$$
or a cone
$$ K_{R,b} = \{ x=(x' , x_N) \in \R^N: |x'| < \frac Rb (b-x_N) \, ; \, 0<x_N<b \}
$$

\end{comment}

\medskip

The symmetry we will get for solutions of \eqref{genprob}  in cylindrical domains in $\R^N$, $N \geq 3$, is a variant of the axial symmetry known as \emph{foliated Schwarz symmetry} considered in several previous papers in connection with Dirichlet problems (see  \cite{BaWi},  \cite{DaPa}, \cite{DaGlPa},  \cite{GlPaWe},  \cite{Pa},  \cite{PaRa} , \cite{PaWe}, \cite{SmWi} and the references therein), whose definition we recall in Section 4. \par
 We will call it sectional foliated Schwarz symmetry. Since it is meaningful for $N \geq 3 $, we will not consider the case $N=2$.

 \begin{definition}\label{sectionfoliatedSS}
 Let $\Omega $ be a bounded  domain with cylindrical symmetry in $\R^N$, $N \geq 3$, and let $u: \Omega \to \R $ a continuous function.
 We say that $u$ is \emph{sectionally foliated Schwarz symmetric} if there exists a vector $p'= (p_1, \dots , p_{N-1},0) \in \R^N$,  $|p'=1|$, such that 
 $u(x)=u(x', x_N )$ depends only on $x_N$, $r= |x'|$ and $\gth = \arccos (\frac {x'}{|x'|} \cdot {p'})$ and $u$ is nonincreasing in $\gth$.
 \end{definition}
 \smallskip
 The definition just means that the functions $x' \mapsto u(x', h)$ are either radial for any $h \in (0,b)$, or nonradial but foliated Schwarz symmetric for any 
 $h \in (0,b)$ in the corresponding domain $\Omega ^h=\Omega  \cap \{ x_N=h \}$, with the same axis of symmetry. \par
 The symmetry result we prove is the following.
 
 \begin{theorem} \label{MainTheorem}
 Let $\Omega $ be a bounded domain in $\R^N$, $N \geq 3$, which is cylindrically symmetric
 as in Definition \ref{DefDominiSimmetrici}, and $\Gamma _2 $ and $\Gamma _1$ described as in \eqref{FrontieraDominiSimmetrici}. \par
 Let $u \in H^{1}_0 (\gO \cup \Gamma _2) \cap C^1 (\ov \Omega) $ be a weak solution of \eqref{genprob}, where $f$ and $g$ satisfy \eqref{Regolarita-f,g} and have the form 
  $f(x,s)=f(|x'|,x_N,s)$,
 $g(x',s)=g(|x'|,s)$ (i.e. they depend on $x'$ through the modulus $|x'|$). \par
 Assume further that $f$ and $g$ are strictly convex in the $s$- variable and that $u$ has Morse index $\gm (u) \leq N-1$.\par
 The $u$ is sectionally foliated Schwarz symmetric.
 \end{theorem}
 
 \begin{remark}
 An analogous result holds for the Dirichlet problem in cylindrically symmetric domains, see Theorem \ref{DirichletProblems}.
 \end{remark}
 
 The definition of Morse index will be recalled in Section 3. \par
 Note that, since $N \geq 3$, Theorem \ref{MainTheorem} applies in particular to solutions with Morse index $1$ or $2$, which can be obtained by variational methods (Mountain Pass or constrained minimization) for many superlinear problems. \par
 One of the ingredients to prove Theorem \ref{MainTheorem} is the maximum principle, in particular we will use it in the weak version for domains with small measure, that we derive in Section 2  as a consequence of some Poincar\' e trace inequality in the space $H^{1}_0 (\gO \cup \Gamma _2)  $. \par
  In order to exploit  the information on the Morse index of the solution to get its symmetry, it is important to be able to characterize it as the number of negative eigenvalues of an associated  linear operator. \par
  It turns out that a good eigenvalue problem to consider to this aim is the mixed boundary conditions eigenvalue problem
   \be \label{ProblemaAutovaloriInt}
\begin{cases}
  - \gD w_j + c(x)  w_j =   \gl _j w_j  &  \text{ in }  \gO \\
   w_j =0  & \text{ on } \Gamma _1 \\
    \frac {\partial w_j}{\partial \gn}+ d(x) w_j=  
 \gl _j w_j  & \text{ on } \Gamma _2
 \end{cases}
 \ee
In section 3 we construct and provide the variational characterization of the eigenvalues of this problem, following the usual approach for the Dirichlet problems, but working in the product Hilbert space
$L^2 (\Omega ) \times L^2 (\Gamma _2)$ (see Theorem \ref{varformautov}). \par
We believe that this construction is interesting in itself. \par
Note that \eqref{ProblemaAutovaloriInt} is related to some weighted  eigenvalue problem, that has been considered in the literature.
 In particular in the interesting paper \cite{Ma} (see also the references therein), 
 although a more general problem with weights is considered, the coefficients $c(x)$ and $d(x)$ in \eqref{ProblemaAutovaloriInt} are supposed to be nonnegative, while, dealing with linearized operators of semilinear elliptic problems, in which case $c= - \frac {\partial f}{\partial s}$, $d= - \frac {\partial g}{\partial s}$, this assumption is not reasonable, and will not be assumed by us (see Remark  \ref{CfrAltriAutori} for a more detailed comment). \par
We also would  like to point out that if we were studying harmonic functions (i.e. if $f \equiv 0 $ in \eqref{genprob}) then another eigenvalue problem could be considered, namely 
\eqref{problemaautovalori3}, in order to characterize the Morse index of a solution (see Remark \ref{AltriAutovalori}). \par
 The paper is organized as follows. \par
 In Section 2 we show some Poincar\' e trace inequality and derive some maximum principle. \par
 In Section 3 we present the spectral theory for the eigenvalue problem \eqref{ProblemaAutovaloriInt} and characterize the Morse index of a solution of \eqref{genprob}. \par
 Finally in Section 4 we prove the symmetry result stated in Theorem \ref{MainTheorem}.

\section{Integral inequalities and Maximum Principles}
Let $\Omega $ be a bounded domain in $\R^N$, $N \geq 2$, with its subsets $\Gamma _1$, $\Gamma _2$ of the boundary , as described in 
\eqref{GeneralBoundary1},
\eqref{GeneralBoundary2}. For a mixed boundary condition linear problem we have

\begin{theorem}[Strong Maximum Principle] \label{Strong Maximum Principle}
Let  $v \in H^{1}_0 (\gO \cup \Gamma _2) \cap C^1 (\ov \Omega) $ be a weak solution of 
\begin{equation} \label{SMP} 
\begin{cases} - \gD v + c(x) v \geq 0 \quad &\text{in } \gO \\
v \geq 0 \quad & \text{in } \gO \\
v= 0 \quad & \text{on } \Gamma _1\\
\frac {\partial v} {\partial \gn}+ d(x)v \geq 0 \quad & \text{on }  \Gamma _2 
\end{cases}
\end{equation}
with $c \in L^{\infty}(\Omega)$, $d \in C^0(\Gamma _2)$. \par
Then either $v \equiv 0 $ in $\Omega$ or $v>0 $ in $\Omega \cup \Gamma _2$.
\end{theorem}
\begin{proof} By the classical strong maximum principle (see e.g. \cite{GiTr}), if $v \not \equiv 0 $ in $\Omega $ then $v>0 $ in $\Omega $
and hence, 
  by continuity,  $v \geq 0 $ on $\Gamma _2 $. \par
 Let  $x_0 \in \Gamma _2 $ and suppose by contradiction that $v(x_0)=0 $. 
 Then $\frac {\partial v } {\partial \nu }(x_0)<0 $ by Hopf's Lemma, since $v$ is positive in $\Omega $ and vanishes in $x_0$. This contradicts the Neumann condition 
 $\frac {\partial v} {\partial \gn}(x_0)+ d(x)v(x_0)=\frac {\partial v} {\partial \gn}(x_0)  \geq 0$. \par
 So $v$ is positive on $\Gamma _2$.
\end{proof}

\par
\medskip

We recall now some well known inequalities in the half spaces (see any book dealing with Sobolev Spaces, e.g. \cite{Br}, \cite{Ev}, \cite{Ke}).
Let us set \par
$\R^N_{+}= \{ x=(x_1, \dots , x_N) \in \R^N:  x_N>0 \} 
$ , \par  
$\R^N_0=  \{ x=(x_1, \dots , x_N) \in \R^N:  x_N=0 \} = \partial \R^N_+
$.

\begin{theorem}[Sobolev and Trace inequalities in $\R^N_+$] 
If $N \geq 3 $  there exist  constants $C_1, C_2 >0 $ such that 
\begin{equation} \label{SobolevInequality} 
(\int _{\R^N_+} |v|^{\frac {2N} {N-2}} \, dx)^{\frac {N-2}{2N} } \leq C_1  ( \int _{\R^N_+} |\nabla v| ^2 \, dx) ^{\frac {1} {2}}
\end{equation}
\begin{equation} \label{TraceInequality} 
 \; (\int _{\R^N_0= \partial \R^N_+} |v|^{\frac {2N-2} {N-2}} \, dx)^{\frac {N-2}{2N-2} } \leq C_2  ( \int _{\R^N_+} |\nabla v| ^2 \, dx) ^{\frac {1} {2}}
\end{equation}
for any
$v \in H^1 (\R^N_+  ) $ (where in the last inequality the value of $v$ on $\Gamma _2$ is to be understood in the sense of traces of functions in Sobolev Spaces).
% \par 
% If $N=2$ the norm in the left can be substituted with any $L^q$ norm.
\end{theorem}

\begin{comment}
\begin{proof} It suffices to prove the inequalities for any
$v \in C_c^1 (\overline {\R^N_+}  ) $, since this latter space is dense in $H^1 (\R^N_+  ) $. \par
The proof of this and of both the inequalities is only a slight modification of the usual proof for function defined in the whole space $\R^N $, for which we refer to 
 \cite{Br}, \cite{Ev}, \cite{Ke} \par
In fact if $v \in C_c^1 (\overline {\R^N_+}  ) $  we have that \\
$v(x_1, \dots , x_N)= - \int _{x_i }^ {+ \infty } \frac {\partial v} {\partial x_i} \,dx_i $, $i=1, \dots ,N $, and \\ 
$v(x_1, \dots ,0)= - \int _{0}^ {+ \infty } \frac {\partial v} {\partial x_N} \,dx_N $ \dots \\
\end{proof}

\end{comment}

Let us now consider  a cylindrically symmetric bounded domain $\Omega $ in $\R^N$, $N \geq 3$, with its subsets $\Gamma _1$, $\Gamma _2$ of the boundary, as described in Definition \ref{DefDominiSimmetrici} and \eqref{FrontieraDominiSimmetrici}.
\par
We take advantage of the simple geometry of our  domains, and prove all the relevant inequalities that we need starting from \eqref{SobolevInequality} and \eqref{TraceInequality}. Of course many of the results hold in a much more general setting (see Remark \ref{DisugPiuGenerali}). \par
If $v \in H_0^1(\Omega \cup \Gamma _2) $ then the the trivial  extension of $v$ to $\R^N_+$ belongs to $H^1 (\R^N_+ )$ and has vanishing trace on 
$\R^N_0 \setminus \Gamma _2$.  \par
As a consequence, using H\" older's inequality, we obtain Poincar\' e's type inequalities both in $\Omega $ and on the flat boundary $\Gamma _2$. \par
 More precisely we have the following

\begin{theorem}[Poincar\' e's inequalities in $H_0^1(\Omega \cup \Gamma _2)$] \label{Poincare's}
Let $N \geq 2$ and  let  $\Omega $ be a cylindrically symmetric domain. There exist  constants $C_1, C_2 >0 $ such that for any
$ v \in H_0^1(\Omega \cup \Gamma _2)$ 
\begin{equation} \label{PoincareInequality1} 
\int _{\Omega} |v|^{2} \, dx \leq C_1 \,  ( \text{meas}_N \; [v \neq 0 ] )^{\frac 2N} \int _{\Omega} |\nabla v| ^2 \, dx
\end{equation}

\begin{equation} \label{PoincareInequality2} 
\int _{\Gamma _2} |v|^{2} \, dx \leq C_2 \, ( \text{meas} _{N} \; [v \neq 0 ] )^{\frac 1{N}} \int _{\Omega} |\nabla v| ^2 \, dx
\end{equation}
where  $[v \neq 0 ] = \{ x \in \Omega  : v(x) \neq 0 \}$.

\end{theorem}
\begin{proof}  By density we can assume  that $v \in C_c^{\infty}(\Omega \cup \Gamma _2)$ and
we   denote by $v$ also the trivial extension to $\R^N_+ $.\par
 By H\" older's and Sobolev inequalities we have that 
$$
\int _{\Omega} |v|^{2} \, dx =\int _{[v \neq 0 ]} |v|^{2} \, 1  \, dx    \leq 
(\int _{\R^N_+} |v|^{ \frac {2N} {N-2} } \, dx) ^  {\frac {N-2}N} ( \text{meas}_N \; [v \neq 0 ] )^{\frac 2N} \leq  
$$
$$ C_1 \,  ( \text{meas}_N \; [v \neq 0 ] )^{\frac 2N} \int _{\R^N_+} |\nabla v| ^2 \, dx =
C_1 \,  ( \text{meas}_N \; [v \neq 0 ] )^{\frac 2N} \int _{\Omega} |\nabla v| ^2 \, dx
$$ 
and we get \eqref{PoincareInequality1}. \par
To get \eqref{PoincareInequality2} we observe that for any $x'= (x_1, \dots , x_{N-1}) \in \R^{N-1}$  we have that 
$$
v^2(x',0)= - \int _0 ^{+ \infty} \frac {\partial v^2} {\partial x_N}(x',t) \, dt= 
-2 \int _0 ^{+ \infty} v(x',t) \frac {\partial v} {\partial x_N}(x',t) \, dt  
$$
Integrating over  $\Gamma _2 $ and using the Poincar\' e's inequality \eqref{PoincareInequality1} and H\" older's inequality, we get 
$$
\int _{\Gamma _2}v^2(x',0) \, dx' \leq  2 \int _{\R^N_+}|v||\nabla v| \,dx \leq 
2 (\int _{\Omega}|v|^2)^{\frac 12} (\int _{\Omega}|\nabla v|^2)^{\frac 12}  \,dx \leq 
$$
$$ C_2 ( \text{meas}_N \; [v \neq 0 ] )^{\frac 1N} \int _{\Omega} |\nabla v| ^2 \, dx
$$ \par
\end{proof}
\par
\medskip

\begin{comment}

\begin{remark}  Using H\" older's inequality in $\R^{N-1}$ together with the trace inequality, we can prove in the same way the inequality
\begin{equation} \label{PoincareInequality3} 
 \; \int _{\Gamma _2} |v|^{2} \, dx \leq C_2 \, ( \text{meas} _{N-1} \; [v \neq 0 ] )^{\frac 1{N-1}} \int _{\Omega} |\nabla v| ^2 \, dx
\end{equation}
where  now $[v \neq 0 ] = \{ x \in \Gamma _2 : v(x) \neq 0 \}$ and in the left we denote by $v$ the trace of the function $v$ on $\Gamma _2$.\par
\end{remark}

\end{comment}

\begin{remark} \label{DisugPiuGenerali} It is well known that a pure Sobolev inequality (or trace inequality) in $H^1 (\Omega)$, i.e. an inequality where in the right hand side there is only  the $L^r $ norm of the gradient, it is not true in general (e.g. in bounded domains, since the constant functions belongs to the space). 
On the other hand it is well known that a pure Sobolev inequality and Poincar\' e 's inequality hold in 
$H_0^1 (\gO \cup \Gamma _2)$, provided $\Omega $ is a bounded domain and $\Gamma _1 $ has a positive $(N-1)$-dimensional Hausdorff measure (see e.g. 	\cite{Ke}, \cite{LiPaTr} and the references therein).
\end{remark}
\par 
\medskip
\begin{proposition}[Weak maximum principle in small domains] \label{massimodebole}
Let $N \geq 2$ ,  $\Omega $  a cylindrically symmetric domain,
 $\Omega ' \subseteq \Omega $, $\ga  \in L^{\infty}(\Omega ')$,  $\gb  \in L^{\infty}(\partial \Omega ' )$ with 
$\Vert  \ga \Vert _{L^{\infty }(\Omega ')} \leq M $, $\Vert  \gb \Vert _{L^{\infty }(\partial \Omega ')} \leq M $   and  $v \in H^1 (\Omega ') $.  
Assume that
\begin{equation} \label{maximum} 
\begin{cases} - \gD v \leq  \ga(x) v  \quad &\text{in } \gO '\\
v \leq 0 \quad & \text{on }  \Gamma _1 ' = \partial \Omega ' \cap \R^N_{+} \\
\frac {\partial v} {\partial \gn} \leq  \gb (x)\,v  \quad & \text{on } \Gamma _2 ' =  \partial \Omega ' \cap \R^N_{0}  
\end{cases}
\end{equation}
(this means that $v^+  \in H^1 _0 (\Omega ' \cup \Gamma _2 ') $ and  
$\int _{\Omega '} \nabla v \cdot \nabla \varphi  \leq \int _{\Omega '} \ga (x) v \varphi + \int _{\Gamma _2 '} \gb (x) v  \varphi $
for any $\varphi \in H^1 _0 (\Omega ' \cup \Gamma _2 ') $, $\varphi \geq 0 $). \par
Then there exists $\delta >0 $, depending on $M$,  such that if $\text{meas}_N \; ([v>0] ) < \delta $ then $v \leq 0 $ in $\Omega $.\par
Here $[v>0]= \{x \in \Omega ' : v(x)>0  \}$, and the conditions  is satisfied in particular if $\text{meas}_N (\Omega ') < \delta $.
\end{proposition}

\begin{proof} Let us denote by $v$ also the trivial extension to $\Omega $, which belongs to  $H^1 _0 (\Omega \cup \Gamma _2 )$.
% Since $c(x)$ and $d(x)$ are bounded, there exists $M >0 $ such that
%  $|c(x)| \leq M$ in $\Omega ' $,  $|d(x)| \leq M $ on $\Gamma _2  ' $.\par
By hypothesis the nonnegative function $v^+ $ belongs to $H^1 _0 (\Omega \cup \Gamma _2)$ and can be used as a test function, yielding 
$$
\int _{\Omega } |\nabla v^+|^2 \leq  M  ( \,  \int _{\Omega }  (v^+)^2 +  \int _{\Gamma _2 } (v^+)^2 \, ) 
$$
On the other hand by the Poincar\' e's inequalities \eqref{PoincareInequality1}, \eqref{PoincareInequality2} we get 
$$
 M  ( \,  \int _{\Omega }  (v^+)^2 +  \int _{\Gamma _2 } (v^+)^2 \, ) \leq 
 $$
$$
 M  \, C \,  [ \; (\text{meas}_N \; ([v>0] ))^{\frac 2N}+(\text{meas}_N \; ([v>0] ))^{\frac 1N} ] \;
 \int _{\Omega } |\nabla v^+|^2 
$$
If the measure of $([v>0] ) $ is sufficiently small then $ M  \, C \,  [ \; (\text{meas}_N \; ([v>0] ))^{\frac 2N}+(\text{meas}_N \; ([v>0] ))^{\frac 1N} ] <1 $, which implies that $v^+ \equiv 0 $ in $\Omega $.
\end{proof}

\par 
\medskip

 \section{Morse index and Spectral theory}
 In this section we consider  mixed boundary problems in general bounded domains, i.e.  we suppose that 
  $\Omega $ is a bounded domain in $\R^N$, $N \geq 2$,  and $\Gamma _1$, $\Gamma _2 $ are relatively open nonempty disjoint subset of the boundary $\partial \Omega$ that satisfy \eqref{GeneralBoundary1}, \eqref{GeneralBoundary2}.
  \par
     \smallskip

Let us recall the following definition, for solutions of \eqref{genprob}.
\begin{definition} Let $u$ be a $C^1 (\gO \cup \Gamma _2)$ solution of \eqref{genprob}. 
\begin{itemize}
\item [i)]
We say that $u$ is  stable (or that has zero Morse index) if the quadratic form 
\begin{equation} \label{formaquadraticalinearizzato} 
Q_u (\psi ; \gO ) = \int _{\gO}  |\nabla \psi |^2 - \int _{\gO} \frac {\partial f}{\partial u} (x, u)|\psi |^2  dx -
\int _{\Gamma _2}  \frac {\partial g}{\partial u} (x, u)|\psi |^2 dx'   
\end{equation}
satisfies  $Q_u (\psi ; \gO )  \geq 0$ for any $\psi  \in C_c^1 (\gO \cup \Gamma _2 )$.
\item [ii)] $u$ has  Morse index  equal to the integer $\mu=\mu (u)\geq 1$ if $\mu $ is the maximal dimension of a subspace of  $ C_c^1 (\gO \cup \Gamma _2 )$ where the quadratic form is negative definite.
\item [iii)] $u$ has infinite Morse index  if for any integer $k$ there is a $k$-dimensional subspace of  $ C_c^1 (\gO \cup \Gamma _2 )$ where the quadratic form is negative definite.
\end{itemize}
\end{definition}
In general, to handle the definition it is convenient to relate the Morse index to the number of negative eigenvalues of a suitable linear eigenvalue problem.
We consider the following one
 $$ 
\begin{cases}
  - \gD w_j + c(x)  w_j =   \gl _j w_j  &  \text{ in }  \gO \\
   w_j =0  & \text{ on } \Gamma _1 \\
    \frac {\partial w_j}{\partial \gn}+ d(x) w_j=  
 \gl _j w_j  & \text{ on } \Gamma _2
 \end{cases}
 $$
 with 
 $$
 c= - \frac {\partial f }{\partial u} \quad , \quad d= - \frac {\partial g }{\partial u}
 $$
% but in some cases the eigenvalues can be put only in the equation in $\Omega $, or in the nonlinear boundary condition (see the remarks that follow the definition of eigenvalues). \par

\begin{remark}\label{CfrAltriAutori}
Some other eigenvalue problems with weights have been considered in the literature (see  \cite{Au},  \cite{GaRoSa}, 
 \cite{Ma} and the references therein). In particular in   \cite{Ma}  the eigenvalue problem
  $$ 
\begin{cases}
  - \gD w_j + c(x)  \, w_j =   \gl _j \, m(x) \,w_j  &  \text{ in }  \gO \\    
  \frac {\partial w_j}{\partial \gn}+ d(x) \, w_j=  
 \gl _j \, n(x) \,w_j  & \text{ on } \partial \Omega 
 \end{cases}
 $$
 with positive weights $m$, $n$ is considered and the eigenvalues sequence constructed by constrained minimization. \par
  In that paper the weights can also vanish in part of the domain, but
the nonnegativity of the coefficients $c$, $d$ is assumed (or more generally coercivity of the corresponding operator), while
in the applications to nonlinear problems as \eqref{genprob} many choices of nonlinearities $f$ and $g$ can lead instead to negative or sign changing coefficients and noncoercive operator (e.g. the case of many superlinear problems).
Therefore when dealing with Morse index properties we prefer to consider problems without weights but with possibly negative or sign changing coefficients (having bounded negative parts). \par
\end{remark}

\medskip
We  now construct the eigenvalues sequence and prove the variational characterization of the eigenvalues following the standard methods used for Dirichlet problems (based on the theory of positive compact selfadjoint operators) by working in the product space $L^2 (\Omega ) \times L^2 (\Gamma _2)$ (we will give all the details in the sequel).\par
\smallskip
Since  $\Gamma _1  $ has a positive $(N-1)$ dimensional Hausdorff measure, the scalar product  in the  Hilbert space
 $\huno $, can be  defined by  
 $$
(f,g)_{\huno}=  \int _{\gO} \nabla f\,\cdot \, \nabla g \, dx
$$
We will denote by the same symbol a function belonging to $\huno $ and its trace on the boundary. \par
Let us consider the bilinear form in $ \huno$ defined by
 \be \label{bilinearform} B(u,\gf)= \int _{\gO} \left [ \nabla u \cdot \nabla \gf + cu \gf \right ]
 + \int _{\Gamma _2}  du \gf  
 \ee
 where we suppose that 
 \be \label{ipotesi coefficienti1}
 c^- \in L^{\infty}(\Omega) \; , \;d^- \in L^{\infty}(\Gamma _2)  
 \ee
 and 
  \be \label{ipotesi coefficienti2}
 c \in  L^{\frac N2} ( \gO ) \; , \;d \in  L^{N-1} ( \Gamma _2) 
 \ee
 if $N \geq 3$, while $c$ and $d$ can belong to any $L^q $ space, $q \geq 1$, if $N=2$. \par
\medskip
  Let us  define,
  together with the bilinear form $B$ defined in \eqref{bilinearform}, the bilinear form
  \be \label{bilinearformLambda} B^{\gL}(u,\gf)= \int _{\gO} \left [ \nabla u \cdot \nabla \gf + (c+ \Lambda )u \gf \right ]
 + \int _{\partial \gO}  (d+ \Lambda )u \gf  
 \ee
 for $\gL \geq 0 $.    \par
 Since \eqref{ipotesi coefficienti1} and \eqref{ipotesi coefficienti2} hold, $B$ and $B^{\gL}$ are continuous symmetric bilinear forms on $\huno$, and there exists $\gL \geq 0 $ such that $ B^{\gL}$ is coercive in  $ \huno $, i.e. it is an equivalent scalar product in $\huno $.\par
 \smallskip
Let us  consider the Hilbert space $\bold V=L^2(\Omega) \times L^2 (\Gamma _2)$, with the scalar product
$  (\,(f_1,f_2)\cdot(g_1,g_2)  \,) = \int _{\Omega } f_1\, g_1 \, dx + \int _{\Gamma _2 } f_2 \, g_2 \, dx' $. \par
  We identify $f=(f_1,f_2) \in \bold V $ with the continuous linear functional 
 \be \gf \in \huno \mapsto (f_1,\gf)_{\ldue (\Omega)}+ (f_2,\gf)_{\ldue (\Gamma _2) }
 \ee
 where as before $\gf$ denote also the trace of the function
   on $\Gamma _2$.\par 
 \medskip
 By the Riesz representation theorem,
 for any $f=(f_1,f_2) \in \bold V $ there exists a unique $u=: T\, f \in \huno $ such that  
 $$
 B^{\gL}(u,\gf) = (f,\gf)_{\bold{V} }=(f_1,\gf)_{\ldue (\Omega)}+ (f_2,\gf)_{\ldue (\Gamma _2) }\; \forall \;\gf \in \huno
 $$
  i.e. $u$ is the unique weak solution of the problem
  \begin{equation} \label{equazioneoperatore}
\begin{cases} - \gD u + \left (c(x) + \Lambda   \right ) u = f_1 \quad &\text{in } \gO \\
u= 0 \quad & \text{on } \Gamma _1 \\
\frac {\partial u}{\partial \gn}+ (d(x)+ \Lambda ) u= f_2  \quad &\text{on } \Gamma _2
\end{cases}
\end{equation}
The solution $u$ belongs to $\huno$ and
$$
\Vert u \Vert _{\huno} \leq c \Vert f \Vert _{\bold{V}} 
$$ 
\par
If we identify a function $u \in \huno $ with the couple $(\,u, \text{ Trace }(u)\,)\in \bold V $, we can consider $T$ as a continuous linear operator $T: \bold V \to \bold V $ defined by
  $ f=(f_1, f_2) \mapsto (\,u, \text{ Trace }(u)\,) $, which maps  $\bold V $ into $ \bold V $.\par
  $T$  is \emph{compact} because of the compact embedding of $\huno $ in $\bold V $.
% xxx
\begin{comment}
VERIFICARE BENE TUTTO \par
IN ALTERNATIVA, SENZA USARE LE STIME $H^2$ SI PU\` O DIMOSTRARE DIRETTAMENTE CHE $T$ \`E  COMPATTO ? AD ES COSI':\par
if $f_n$ is bounded in $\huno $ then, up to a subsequence, it  converges weakly in $\huno$ and strongly in $L^2$ to a function $f \in \huno$. The corresponding sequence of the solutions $u_n= T f_n$ is bounded in $\huno $ and it converges weakly to a solution $u$ of the problem \eqref{equazioneoperatore}, i.e.
$B^{\gL}(u,\phi) = (f,\phi)_{\ldue (\Omega)}+ (f,\phi)_{\ldue (\partial \Omega) }$ for any $\phi \in \huno$. Then
$\Vert u-u_n \Vert _{\huno}^2=B^{\gL}(u-u_n,u-u_n) = (f-f_n,u-u_n)_{\ldue (\Omega)}+ (f-f_n,u-u_n)_{\ldue (\partial \Omega) }
\leq \Vert f-f_n \Vert  _{\ldue (\Omega)} \Vert u-u_n \Vert  _{\ldue (\Omega)}+
\Vert f-f_n \Vert  _{\ldue (\partial \Omega)} \Vert u-u_n \Vert  _{\ldue (\partial \Omega)} \leq C  ( \Vert f-f_n \Vert  _{\ldue (\Omega)}+ \Vert f-f_n \Vert  _{\ldue (\partial \Omega)} ) \to 0$ 
\par
\medskip
\end{comment}
% xxx
 Moreover 
 it is a \emph{positive} operator, since (recall that $ B^{\gL} $ is an equivalent scalar product in $\huno $ ) \par
 \smallskip
  $
   (Tf,f)_{ \bold{V} }=(u,f)_{\bold{V}} =B^{\gL}(u,u) >0  $ if $f=(f_1, f_2) \neq 0$, so that $u \neq 0 $, \par
    \smallskip
 \noi  and it is also \emph{selfadjoint}. \par
  Indeed 
 if $Tf=u$, $Tg=v$, i.e \par
 \smallskip
 $B^{\gL}(u,\phi) = (f,\phi)_{\bold{V}  }$,   
 $B^{\gL}(v,\phi) = (g,\phi)_{\bold{V} } $,   
  then \par
  \smallskip
 $(Tf,g)_{\bold{V} }=(u,g)_{\bold{V}}=(g,u)_{\bold{V}} =B^{\Lambda}(v,u)=B^{ \Lambda}(u,v)=  $ \par
 \smallskip
 $ (f,v)_{\bold{V}  } = 
 (f,Tg)_{\bold{V} }
  $.
  \par
  \smallskip
   Thus, by the spectral theory of positive compact selfadjoint operators in Hilbert spaces there exist a nonincreasing sequence $\{ \gm _j^{\gL} \} $ of positive eigenvalues with 
 $\lim _{j \to \infty}\gm _j^{\gL}=0 $ and a corresponding sequence $\{ w_j \} \subset \huno $ of eigenvectors such that $T(w_j)= \gm _j^{\gL} w_j $ and $w_j$ is an orthonormal basis of $ \bold V $.\par
 Putting $  \gl _j^{\gL} = \frac 1 { \gm _j^{\gL}  }$ then $w_j$ solve the problem
 $$
 \begin{cases}
  - \gD w_j + (c(x) + \gL ) w_j =   \gl _j^{\gL} w_j  &  \text{ in }  \gO \\
   w_j =0  & \text{ on } \Gamma _1 \\
    \frac {\partial w_j}{\partial \gn}+ (d(x)+ \Lambda ) w_j=  
 \gl _j^{\gL} w_j  & \text{ on } \Gamma _2
 \end{cases}
 $$
 Translating, and denoting by $\gl _j$ the differences $\gl _j=  \gl _j^{\gL} -\gL  $, we conclude that there exist a sequence $\{\gl _j \}$ of eigenvalues, with
 $- \infty < \gl _1 \leq \gl _2 \leq \dots $, $\lim _{j \to + \infty} \gl _j = + \infty $, and a corresponding sequence of eigenfunctions $\{ w_j \}$ that weakly solve the systems
 \be \label{problemaautovalori}
\begin{cases}
  - \gD w_j + c(x)  w_j =   \gl _j w_j  &  \text{ in }  \gO \\
   w_j =0  & \text{ on } \Gamma _1 \\
    \frac {\partial w_j}{\partial \gn}+ d(x) w_j=  
 \gl _j w_j  & \text{ on } \Gamma _2
 \end{cases}
 \ee
Moreover by elliptic regularity theory  the eigenfunctions $w_j $ belong at least to $  C^{1}(\gO)$.\par
\medskip

We now  collect in the next theorem  the variational formulation and some properties of eigenvalues and eigenfunctions.\\
%  In what follows if $\gO '$ is a subdomain of $\gO $ we denote by $\gl _k (\gO ')$ the eigenvalues of the same problem with $\gO $ substituted by $\gO '$.
 
\begin{theorem}\label{varformautov} Suppose that  \eqref{ipotesi coefficienti1} and \eqref{ipotesi coefficienti2} hold. \par
There exist sequences of eigenvalues $\{\gl _j \}_{j \in \N} \subset \R $, with $\lim _{j \to \infty} \lambda _j = + \infty $, and eigenfunctions $\{ w_j \}_{j \in \N} \subset \huno$ that satisfy \eqref{problemaautovalori}. \par
The sequence
$\{ (w_j, \text{Trace}(w_j) ) \}$ is an orthonormal basis of the space $\textbf{V}= L^2(\Omega) \times L^2 (\Gamma _2)$. \par
Then defining the Rayleigh quotient
\be R(v)= \frac {B(v,v)}{ (v,v)_{\ldue (\Omega)} + (v,v)_{\ldue (\Gamma _2)}} \quad \text{ for } v \in \huno \quad v \neq 0
\ee
with $B(.,.)$ as in \eqref{bilinearform}, 
 the following properties hold, where  $\textbf{V}_k$ denotes a $k$-dimensional subspace of $\huno$ and the orthogonality conditions  $v \bot w_k $  or  $v \bot \textbf{V}_k $  stand for the orthogonality in $\textbf{V}$.  
\begin{itemize}
\item[i)] $\gl _1 = \text{ min } _{v\in \huno \,, \, v \neq 0} R(v) =  \text{ min } _{v\in \huno \, , \,  (v,v)_{\bold{V}}=1 } B(v,v) $
\item[ii)]   $\gl _m = \text{ min } _{v\in \huno \,, \, v \neq 0 \, , \, v \bot w_1,\dots , v \bot w_{m-1}} R(V)$\\ $ =  \text{ min } _{v\in \huno \, , \,  (v,v)_{\bold{V}}=1 \, , \, v \bot w_1,\dots , v \bot w_{m-1} } B(v,v) $  if $m \geq 2$ 
\item[iii)] $\gl _m = \text{ min } _{\textbf{V}_m }   \text{ max } _{ v \in \textbf{V}_m \, , \, v \neq 0 } R(v) $
\item[iv)] $\gl _m = \text{ max } _{\textbf{V}_{m-1} }   \text{ min } _{ v \bot  \in \textbf{V}_{m-1 }\, , \, v \neq 0 } R(v) $
\item[v)] If $w \in \huno$, $w \neq 0$, and $R(w)= \gl _1$, then $w$ is an eigenfunction corresponding to $\gl _1$.
\item[vi)] If  $w $ is a first eigenfunction, then $w^{+}$ and $w^{-}$ are eigenfunctions, if they do not vanish.
\item[vii)] The first eigenfunction does not change sign in $\gO $ and  the first eigenvalue is simple, i.e. up to scalar multiplication there is only one eigenfunction corresponding to the first eigenvalue. 
\item[viii)] If $c'(x)\in L^{\infty}(\Omega )$, $d'(x)\in L^{\infty}(\Gamma _2 )$ and 
$c \geq  c'   $, $d \geq  d'   $   then  $ \gl _1 \geq { \gl '} _1$, where $\gl _1 ' $ denotes the corresponding first eigenvalue.
% \item[ix)] $ \lim _{\text{ meas }(\gO ') \to 0} \gl _1 (\gO ') = + \infty $ \par
%( ORA NON HA SENSO QUESTO PUNTO ? )
\end{itemize}
\end{theorem}

\begin{proof}  We just proved the existence of the eigenvalues and eigenfunctions. \par
As before let us consider for   $\gL \geq 0 $ the bilinear form $ B^{\gL}(v,v)= B(v,v) + \gL (v,v)_{\bold{V} }$, which is an equivalent scalar product in $\huno$, and  define $R_{\gL}(V) =\frac {B^{\gL}(v,v)}{ (v,v)_{ \bold{V} } } = \frac {B^{\gL}(v,v)}{ (v,v)_{\ldue (\Omega)} + (v,v)_{\ldue (\Gamma _2)}} \quad \text{ for } v \in \huno \quad v \neq 0 $. 
  Since  $  R_{\gL}(V) = R(V) + \gL $, once the properties are proved for  $  B_{\gL}(V) $ (which is an equivalent scalar product in $\huno $)  and its eigenvalues $\gL _j^{\gL}$, we recover the results stated by translation. Therefore for simplicity of notations we assume from the beginning that $\gL =0$ i.e. that $B(.,.)$ is an equivalent scalar product in $\huno $.\\
The proofs   are   standard  (see e.g. \cite{Ev}, \cite{Ke} for the Dirichlet problem), so we only sketch them.\par
\medskip
i) \ The sequence $\{ w_j \}$ is an orthonormal basis of $\bold{V}  $, and  since $B(w_k,w_j)= \gl _k (w_k,w_j)_{\bold{V} }$ (in particular $=0$ if $k\neq j $), the sequence  
$\{ (\gl _j ) ^{-\frac 12}\, w_j \}$  is an orthonormal basis of $\huno   $. It follows that if $u = \sum _{k=1}^{\infty } d_k w_j $ is the Fourier expansion of a function $u$ in 
$\bold{V} $, the series converges to $u$ in $\huno $ as well. 
If now $(u,u)_{\bold{V} } = \sum _{k=1}^{\infty } d_k^2 =1$, then 
$B(u,u)= \sum _k \gl _k d_k ^2 \geq \gl _1 \sum _k d_k ^2= \gl _1 $ and i) follows.\\
ii) \ If $ v \bot w_1,\dots , w_{m-1} $ and $(v,v)_{\bold{V} }=1)$, then $v = \sum _{k=m}^{\infty } d_k w_j $ and as before $B(v,v) \geq \gl _m$ and since $B(w_m, w_m)= \gl _m$ ii) follows.\\
iii) \ If $\text{ dim }(\textbf{V}_m) =m $ and $\{ v_1, \dots , v_m \}$ is a basis of $\textbf{V}_m  $, there exists a linear combination $0 \neq v=\sum _{i=1}^m \ga _i v_i$ which is orthogonal to $w_1$, \dots $w_{m-1}$ ($m$ coefficients and $m-1$ unknown), so that by ii) we obtain that
$ \text{ max } _{ v \in \textbf{V}_m \, , \, v \neq 0 } R(v) \geq \gl _m$. On the other hand if $\textbf{V}_m= \text{ span }(w_1,\dots , w_m)$ then 
$ \text{ max } _{ v \in \textbf{V}_m \, , \, v \neq 0 } R(v) \leq  \gl _m$, so that iii) follows.\\
iv) \  The proof  is similar. If $\{ v_1, \dots , v_{m-1}\} $ is a basis of an $m-1$-dimensional subspace $\textbf{V}_{m-1}$, there exists a linear combination
$0 \neq w=\sum _{i=1}^m \ga _i w_i$ of the first $m$ eigenfunctions which is orthogonal to $\textbf{V}_{m-1}$, and $R(w,w) \leq \gl _m $.
So  $\text{ min } _{ v \bot  \in \textbf{V}_{m-1 }\, , \, v \neq 0 } R(v) \leq \gl _m $, but taking $\textbf{V}_m= \text{ span }(w_1,\dots , w_{m-1})$ then  
$\text{ min } _{ v \bot  \in \textbf{V}_{m-1 }\, , \, v \neq 0 } R(v) \geq \gl _m $, so that iv) follows.\\
v) \ By normalizing we can suppose that $(w,w)_{\bold{V} }=1$. Let $v \in \huno $, $t>0$. Then by i) $R(w+tv)= \frac { B(w+tv,w+tv)}{(w+tv)_{\ldue}}\geq \gl _1$, i.e.
$B(w,w)+ t^2 B(v,v) + 2t B(w,v) \geq \gl _1 \left [ (w,w)_{\bold{V} }  + t^2 (v,v)_{\bold{V} }  + 2t (w,v)_{\bold{V} }  \right ] = \gl _1 + \gl _1 t^2 (v,v) + 2t \gl _1 (w,v)$. 
Since $B(w,w) = \gl _1$, dividing by $t$ and letting $t \to 0$ we obtain that $B(w,v) \geq \gl _1 (w,v)_{\bold{V} } $ and changing $v$ with $-v$ we deduce that 
$B(w,v) = \gl _1 (w,v)_{\ldue} $ for any $V \in \huno $, i.e. $w$ is a first eigenfunction. \\
vi) \ Multiplying  \eqref{problemaautovalori} by $w_{1}^{+}$ and integrating we deduce that 
$B( w_{1}^{+},w_{1}^{+}) \leq  \gl _1 ( w_{1}^{+}, w_{1}^{+} )$, so that by v) $w_{1}^ {+}$ is a first eigenfunction. The same applies to  $w_{1}^{-}$. \\
vii) \ The conclusion follows from the strong maximum principle. In fact  if $w^{+}$ does not vanish, it is a first eigenfunction by vi), and by the strong maximum principle (Theorem \ref{Strong Maximum Principle}) is strictly positive in $\gO$, i.e. $w >0 $ in $\gO $ if it is positive somewhere.\\
If $w_1$, $w_2$ are two eigenfunctions corresponding to $\gl _1$, they do not change sign in $\gO $, so that they can not be orthogonal in $\ldue $. This implies that the first eigenvalue is simple.\\
viii) Let $w_1$ be  the first eigenfunction for the system \eqref{problemaautovalori}, by normalizing it, we can assume that  $(w_1, w_1)_{\bold{V}}=1$. \\
Denoting by $B '$ the bilinear form corresponding to the coefficients $ c' $, $d'$  we have that
\begin{multline}
 \gl _1 = B(w_1,w_1)=  \int _{\gO}  [ \,    |\nabla w_1|^2   +  c(w_1)^2 \, dx \, ]  + \int _{\Gamma _2} d  (w_1)^2     \geq \\
\int _{\gO}  [ \,    |\nabla w_1|^2   +  c'(w_1)^2 \, dx \, ]  + \int _{\Gamma _2} d'  (w_1)^2     = B' (w_1,w_1) \geq {\gl _1}'
\end{multline}
% ix) \  If $V \in \textbf{H}_0^1 (\gO ')$ there exists $C \geq 0$ such that \par 
% $B_{\gO '} (V,V) \geq \int _{\gO '} |\nabla V |^2 \, dx - C  \int _{\gO '} |V|^2 \, dx \, ]  $ , \par while by  Poincar\'e' s inequality  \par 
% $\int _{\gO '}|V|^2 \, dx \leq C' |\gO '| ^{\frac 2N} \int _{\gO '} |\nabla V |^2 \, dx $,  so that \par 
% $R_{\gO '}(V)= \frac {B_{\gO '} (V,V)}{\int _{\gO '}|V|^2} \geq \frac {  \int _{\gO '} |\nabla V |^2 \, dx } {  \int _{\gO '} |V|^2 \, dx }- C \geq 
% \frac 1{C' |\gO '|^ {\frac 2N}} -C \to + \infty $ if $|\gO ' | \to 0 $.
\end{proof}

 \medskip
 We now consider a solution $u$ of the nonlinear problem \eqref{genprob} and the linearized eigenvalue problem at $u$, namely
 the problem   \eqref{problemaautovalori}, with 
 $
  c(x)= - \frac {\partial f}{\partial u}(x,u(x))
$, 
$ d(x)= - \frac {\partial g}{\partial u} (x,u(x))
$. \par
\medskip
 \begin{theorem} Let $\gO $ be a bounded domain in $\RN$. Then the Morse index of a solution $U$ to \eqref{genprob} equals the  number of negative eigenvalues of the linearized eigenvalue
problem
\be \label{ProbAutovNonLin}
\begin{cases}
  - \gD w_j - \frac {\partial f}{\partial u}(x,u(x)) \, w_j =   \gl _j w_j  &  \text{ in }  \gO \\
   w_j =0  & \text{ on } \Gamma _1 \\
    \frac {\partial w_j}{\partial \gn}- \frac {\partial g}{\partial u}(x,u(x)) \, w_j=  
 \gl _j w_j  & \text{ on } \Gamma _2
 \end{cases}
\ee
   \par

\end{theorem}

\begin{proof} 
 Let us denote by $\mu (U)$ the Morse index as previously defined, and by $m(U)$ the number of negative eigenvalues of \eqref{ProbAutovNonLin}.\\
  If the quadratic form  (defined in \eqref{formaquadraticalinearizzato}) $Q_u$  is negative definite on  a $m$-dimensional subspace of $ C_c^1 (\gO \cup \Gamma _2)$, by   Proposition \ref{varformautov} iii)  the $m$-th  eigenvalue $\gl _m $ is negative, so that $m (U)  \geq \gm (U)$. \\
  On the other hand if there are $m$ negative eigenvalues of problem \eqref{ProbAutovNonLin} in $\gO $, by the continuity of the eigenvalues there exists a subdomain $\gO ' \subset \gO $ 
  where there are $m$ negative eigenvalues and  corresponding orthogonal eigenfunctions $w^1, \dots , w^m$ which by trivial extension can be considered as functions  with compact support in $\gO \cup \Gamma _2 $. Regularizing  these functions  we get that the quadratic form
 $Q_u$ is negative definite on a subspace of $ C_c^1 (\gO \cup \Gamma _2)$ spanned by $m$ linear independent functions, so that $\mu (u)  \geq m(u)$.
\end{proof}

%  VEDERE SE IN QUESTO CONTESTO VA BENE QUESTA DIMOSTRAZIONE O VA AGGIUSTATA. QUI $\Omega $ sarebbe generale, ma se ci si limita ad es. a una semipalla si pu\`o
%  prendere  $\Omega ' = B_{R'}^+ $ con $ R' < R $ VICINO A $R$ \dots  
\par

\medskip

\begin{remark} \label{AltriAutovalori}
If one of the coefficients $c$, $d$ is nonnegative (or more generally if the linear operator associated is coercive on $\huno$ ) then other choices of eigenvalue problem are possible.\par
If e.g. 
 $c \geq 0 $ a modification of the preceeding construction yields a a compact operator in the space $L^2 (\Gamma _2 )$ and a corresponding  sequence of eigenvalues of the problem
  $$
\begin{cases}
  - \gD w_j +c(x)  w_j  =0  &  \text{ in }  \gO \\
   w_j =0  & \text{ on } \Gamma _1 \\
    \frac {\partial w_j}{\partial \gn}+ d(x) w_j =\gl _j '' w_j & \text{ on } \Gamma _2
 \end{cases}
 $$
 This case occurs in particular in the study of harmonic functions  subjected to nonlinear boundary conditions, that has been studied in recent papers (see  e.g. \cite{AbDoOMe}, \cite{BeFoSe} ). In this case $f \equiv 0$ in the nonlinear problem, so that $c \equiv 0 $ in the linearized one, and the previous eigenvalue problem becomes
  \be \label{problemaautovalori3}
\begin{cases}
  - \gD w_j   =0  &  \text{ in }  \gO \\
   w_j =0  & \text{ on } \Gamma _1 \\
    \frac {\partial w_j}{\partial \gn}+ d(x) w_j =\gl _j '' w_j & \text{ on } \Gamma _2
 \end{cases}
 \ee

 \par

\begin{comment}

If  instead $d \geq 0 $ a modification of the preceeding construction yields a compact operator in the space $L^2 (\Omega )$ and a corresponding  sequence of eigenvalues of the problem
 \be \label{problemaautovalori2}
\begin{cases}
  - \gD w_j + c(x)  w_j =   \gl _j ' w_j  &  \text{ in }  \gO \\
   w_j =0  & \text{ on } \Gamma _1 \\
    \frac {\partial w_j}{\partial \gn}+ d(x) w_j =0 & \text{ on } \Gamma _2
 \end{cases}
 \ee

 \end{comment}

In  any case the eigenvalues share the same variational characterization  as the eigenvalues considered by us,  so that  the number of negative eigenvalues is the same and characterize the Morse index of a solution. \par
 The general linear eigenvalue problem \eqref{problemaautovalori} that we consider, with the \emph{same} eigenvalue parameter in the equation and in the nonlinear boundary condition, has the advantage of not requiring the nonnegativity of the coefficients $c$, $d$; moreover the
  existence and characterization of the eigenvalues it is natural to prove in the product space $\bold V$.
 \end{remark}

\section{Proof of Theorem \ref{MainTheorem} }

In this section we will prove Theorem \ref{MainTheorem}, so we will work under the assumptions on $\Omega $, $f$, $g$ and $u$ stated in this theorem. The symmetry for the solution that we are going to prove is the sectional foliated Schwarz symmetry that we have introduced in Definition  \ref{sectionfoliatedSS}. \par
To compare it with the usual foliated Schwarz symmetry let us recall this last definition.
\begin{definition}\label{foliatedSS}
Let $\gO$ be a rotationally symmetric domain in $\RN$, $N\geq 2$. We say that a continuous function $u: \Omega \to \R $ is foliated Schwarz symmetric if there exists a vector $p \in \RN$, 
$|p|=1$, such that $u(x)$ depends only on $r=|x|$ and $\theta= \arccos \left ( \frac {x}{|x|}\cdot p  \right ) $ and $u$ is nonincreasing in $\theta $.\par
\end{definition}
 Let us observe that  for solutions $u$ of  semilinear elliptic equations,  foliated Schwarz symmetry implies that either $u$ is radial or it is strictly decreasing in the angular variable $\theta $ (see e.g. \cite{Pa}, \cite{PaWe}). \par 
The sectional foliated Schwarz symmetry just means to have foliated Schwarz symmetry on any section   $\Omega ^h=\overline {\Omega } \cap \{ x_N=h \}$, $h \in (0,b)$, with respect to the same vector $p'= (p_1, \dots , p_{N-1},0) \in \R^N$,  $|p'=1|$. \par
\smallskip
For the proof of Theorem \ref{MainTheorem} we need to fix some notations and prove some preliminary results.\par
\smallskip

 Let $N \geq 3 $ and let us denote, for simplicity of notations
 $$
 S^{N-2}=\{ \bold e = (e_1, \dots ,e_N):|\bold e |=1 \, ; \, e_N=0 \}
 $$
  i.e. $S^{N-2}$ it is the set of the directions orthogonal to the direction $\bold e _N=(0, \dots,1)$. \par
  For any such direction let us define the hyperplane $T(\bold e)$ and the  ''cap'' $\gO ( \bold e) $ as
$$ T(e)  = \{ x \in \RN : x \cdot e=0\}\; , \quad  \gO (e)= \{ x \in \gO : x \cdot e >0\}  $$
with the corresponding boundaries
$$\Gamma _1 (\bold e)= (\Gamma _1 \cap \overline {\Omega (\bold e)}) \cup (T(\bold e) \cap \Omega (\bold e)) \; ; \;
\Gamma _2 (\bold e)= \Gamma _2 \cap (\overline {\Omega (\bold e)} \setminus T(\bold e) )$$
 Moreover if $x \in \gO $ we denote by $\gs _{e}(x)= x -2 (x \cdot e)e$ the reflection of $x$ through the hyperplane $T(e)$ and by $u_{\gs _{e}}$ the function $u \circ \gs _{e} $ .\par

   We start by proving a sufficient condition for the sectional foliated Schwarz symmetry.
   
 \begin{lemma}  \label{condsuffFSS} Let $\Omega $ be a cylindrically symmetric domain in $\R^N$, $N \geq 3$, as Definition \ref{DefDominiSimmetrici} , and let 
 $u: \Omega \to \R $ be a continuous function.
 \par
Assume that for every direction $\bold e =(e_1, \dots e_N ) \in S^{N-2}$  it holds that
 either $u \leq  u_{\gs (\bold e)}$ in $\Omega (\bold e )$ or $u\geq  u_{\gs (\bold e)}$ in $\Omega (\bold e )$. \par
 Then $u$ is sectionally foliated Schwarz symmetric.
% \item there exists a direction $\bold e =(e_1, \dots e_N ) \in S^{N-1}$  with  $\bold e \cdot \bold e_N =0 $ such that $u$ is symmetric with respect to the hyperplane
% $T_e$ and the first eigenvalue in $\Omega (\bold e)$ of the problem  \eqref{AutovalCappe} is nonnegative.
 \end{lemma}
 
 \begin{proof}
 It follows from an analogous  sufficient condition for the foliated Schwarz symmetry in rotationally symmetric domains, just applied to each section $\Omega ^h $, $h \in (0,b)$, which is either a ball or an annulus in $\R^{N-1}$ (see \cite{Pa}, \cite{PaWe} and the references therein).
 \end{proof}
 Let us also observe that in general if $u$ is a solution of a semilinear elliptic equation, then the sectional foliated Schwarz symmetry of $u$ implies that either 
 $u(.,x_N)$ is radial for every $x_N$ or it is \emph{strictly} decreasing in the angular variable $\theta $ (see the proof of  Theorem \ref{MainTheorem} that follows). \par
 \medskip
 
 As we saw in previous section 
 the Morse index of a solution $u$ to \eqref{genprob}, with  $f(x,s)=f(|x'|, x_N,s)$, $g(x',s)=f(|x'|,s)$,  equals the  number of negative eigenvalues of the linearized eigenvalue problem
\be \label{ProbAutovNonLinSimm}
\begin{cases}
  - \gD w_j - \frac {\partial f}{\partial u}(|x'|, x_N,u(x)) \, w_j =   \gl _j w_j  &  \text{ in }  \gO \\
   w_j =0  & \text{ on } \Gamma _1 \\
    \frac {\partial w_j}{\partial \gn}- \frac {\partial g}{\partial u}(|x'|,u(x)) \, w_j=  
 \gl _j w_j  & \text{ on } \Gamma _2
 \end{cases}
\ee
in the whole domain $\Omega $.\par

 \begin{comment}

POSSIAMO OMETTERE IL PROSSIMO TEOREMA ($f'$ e $g' $ convesse) E ANNUNCIARLO VAGAMENTE IN FUTURO  ? \par
\begin{theorem}\label{simmetriaindicebasso2}  Let $\Omega = B_R^+= B_R \cap \R^N_{+}=  \{ x=(x_1, \dots , x_N) \in \R^N: |x| < R \, ; \, x_N>0 \}$ and suppose that $u \in H_0^1(\Omega \cup \Gamma _2) \cap C^1 (\overline { \Omega } )$ is a weak solution of \eqref{genprob} that has Morse index $\gm (u) \leq N-1$.  \par
Suppose further that $f$ and $g$ are continuous functions in $ [0, + \infty) \times [0, + \infty) $ which are $C^1$ with respect to the second variable,  and  \par
ii) \ \ $ \frac {\partial f}{\partial u}(|x|,.) $ and $ \frac {\partial g}{\partial u}(|x|,.) $  \ are strictly convex functions.\par 
Then $u$ is sectionally foliated Schwarz symmetric.
 \end{theorem}

 \end{comment}
 \smallskip
 
 Now we consider a similar eigenvalue problem but in the caps $\Omega (e)$, and we denote by
 $\gl _j^{e}$ and  $\gf _j^{e}$, $j \in \N$, the corresponding eigenvalues and eigenfunctions:
 
   \be \label{AutovalCappe}
 \begin{cases}
  - \gD \gf _j^{\bold e} - \frac {\partial f}{\partial u}(|x'|,x_N,u)   \gf _j^{\bold e} =   \gl _j^{\bold e}  \gf _j^{\bold e}  &  \text{ in }  \gO (\bold e) \\
   \gf _j^{\bold e} =0  & \text{ on }( \Gamma _1 (\bold e)  ) \cup T(\bold e) \\
    \frac {\partial \gf  _j^{\bold e}}{\partial \gn} - \frac {\partial g}{\partial u}(|x'|,u)  \gf _j^{\bold e}=  
\gl  _j^{\bold e} \gf _j^{\bold e}  & \text{ on } \Gamma _2 (\bold e)
 \end{cases}
 \ee
 
 The  same properties as  for the eigenvalues and eigenfunctions defined by \eqref{ProbAutovNonLinSimm} hold.
 In particular if  we define the quadratic form 
 $$
  Q_u^{\bold e} (\psi  ) =
 \int _{\gO (\bold e)}  |\nabla \psi |^2 - \int _{\gO (\bold e)} \frac {\partial f}{\partial u} (|x'|,x_N, u)|\psi |^2  dx -
\int _{\Gamma _2 (\bold e )}  \frac {\partial g}{\partial u} (|x'|, u)|\psi |^2 dx'   
 $$
for $ \psi  \in 
 H_0^1 (\Omega (\bold e) \cup \Gamma _2 (\bold e))$,  and the Rayleigh qoutient 
 $$
  R^{\bold e}(v)= \frac {Q_u^{\bold e}(v)}{ (v,v)_{\ldue (\Omega (\bold e))} + (v,v)_{\ldue (\Gamma _2 (\bold e))}}
  $$ 
   for $ v \in 
 H_0^1 (\Omega (\bold e) \cup \Gamma _2 (\bold e))$, $ v \neq 0
$,
 then 
 \be \label{primoautovalcappa} \gl _1^{e}= \text{ min } _{v\in  H_0^1 (\Omega (\bold e) \cup \Gamma _2 (\bold e)) \,, \, v \neq 0} \; R^{\bold e}(v) 
  \ee
  Note that
  the first eigenfunction $\gf ^{\bold e}:=\gf _1^{e} $ does not change sign in $\Omega (\bold e)$. \par
 \smallskip
 We have 
 \begin{lemma} \label{esistcappaconautovalpositivo} Let $u$ be a solution of problem \eqref{genprob} with Morse index  $\gm (u) \leq N-1$.
 Then there exists a direction 
 $\bold e \in S^{N-2}$ such that the first eigenvalue $\gl _1 ^{\bold e} \geq 0 $.
 \end{lemma}
\begin{proof}  The proof is immediate if the Morse index of the solution satisfies $\mu (u) \leq 1$. \par
 In fact in this case for any direction $\bold e$ at least one amongst 
$ \gl _1^{e}$  and 
$ \gl _1^{-e}$ must be nonnegative, otherwise taking the  corresponding first eigenfunctions we would obtain a $2$-dimensional subspace of  $H_0^1(\gO \cup \Gamma _2 )$ where the quadratic form $ Q_u^{\bold e} $ is negative definite.  \\
So let us assume that   $2\leq j= \mu (u) \leq N-1 $.\par
Denote by $w_k$ the eigenfunctions of problem \eqref{ProbAutovNonLinSimm}, and 
for any direction $\bold e \in S^{N-2}$
let us consider the function 
$$ \gps ^{\bold e}(x)=
 \begin{cases} 
 \left (  \frac  {  ( \gf _1^{-e} \,, \, w _1 )_{\bold V} }  {  ( \gf _1^{e} \,, \, w _1 )_{\bold V} }  \right ) ^{\frac 12} \gf _1^{e} (x)
  \quad &\text{ if } x \in \gO (e)    \\
  -  \left (  \frac  {  ( \gf _1^{e} \,, \, w _1 )_{\bold V} }  {  ( \gf _1^{-e} \,, \, w _1 )_{\bold V} }  \right ) ^{\frac 12} \gf _1^{-e} (x)
  \quad &\text{ if } x \in \gO (-e)    
\end{cases}
$$
( let us recall that $\bold V = L^2(\Omega ) \times L^2 (\Gamma _2)$ and in the scalar product in the space $\bold V$ we consider the trivial extension to $\Omega $ of the eigenfunctions $\gf ^{\bold e}:=\gf _1^{e}$). \par
 The mapping $e \mapsto \gps _e $ is a continuous odd function from $S^{N-2}$ to $\huno $  and, by construction,  
 $ (\gps ^{\bold e} \,, \, w _1 )_{\bold V} =0$. \\
 Therefore the function $h: S^{N-2}\to \R ^{j-1} $ defined by
 $$ h(\bold e)= \left ( ( \gps ^{\bold e} \,, \, w _2 )_{\bold V},   \dots , ( \gps ^{\bold e} \,, \, w _j ){\bold V} \right )
 $$
 is an odd continuous mapping, and since $j-1 < N-1 $,  by the Borsuk-Ulam Theorem it must  have a zero.
 This means that there exists a direction $e\in S^{N-2}$ such that
   $\gps ^{\bold e} $ is orthogonal to all the eigenfunctions $w _1, \dots , w_j$. Since $ \mu (u)=j $  this implies that 
 $Q_u(\gps ^{\bold e} ; \gO )= B_u(\gps ^{\bold e}, \gps ^{\bold e})  \geq 0 $, which in turn implies that either $Q_u(\gf ^{\bold e} ; \gO (e) ) \geq 0 $ or 
 $Q_u(\gf ^{\bold {-e} } ; \gO (-e) ) \geq 0 $, i.e. either
 $\gl _1^{\bold e} $ or $\gl _1^{-\bold e} $ is nonnegative, so the assertion is proved.

\end{proof}

\begin{proof}[Proof of Theorem \ref{MainTheorem}] For simplicity of notations we first consider the case $N=3$. \par
 Let $\bold e \in S^{N-2}$ the direction found in Lemma \ref{esistcappaconautovalpositivo} such that  $\gl _1 ^{\bold e}  \geq 0 $ and $v= u_{\gs (\bold e)}$ the corresponding reflection of $u$, so that
 $(u-v)^{\pm} \in H_0^1 (\Omega (\bold e) \cup \Gamma _2 (\bold e)) $.
 Since $f$ and $g$ are strictly convex functions, if $u>v$ then $\frac {f(|x|,u)-f(|x|,v)} {u-v}< \frac {\partial f }{\partial u}(|x|,u)$ and
 $\frac {g(|x|,u)-g(|x|,v)} {u-v}< \frac {\partial g }{\partial u}(|x|,u)$. \par
  It follows, multiplying by $(u-v)^+$ the equations satisfied by $u$ and $v$ and subtracting,   that if $(u-v)^+ \not \equiv 0 $, then
\begin{multline} 0= \int _{\gO (\bold e )} |\nabla (u-v)^+|^2 \, dx  - \int _{\gO (\bold e )} \frac {f(|x|,u)-f(|x|,v)} {u-v} | (u-v)^+|^2 \, dx -\\
- \int _{\Gamma _2 (\bold e )} \frac {g(|x|,u)-g(|x|,v)} {u-v} | (u-v)^+|^2 \, dx' > \\ 
 \int _{\gO (\bold e )} |\nabla (u-v)^+|^2 \, dx  - \int _{\gO (\bold e )} \frac {\partial f }{\partial u}(|x|,u) {u-v} | (u-v)^+|^2 \, dx - \\
- \int _{\Gamma _2 (\bold e )} \frac {\partial g }{\partial u}(|x|,u) {u-v} | (u-v)^+|^2 \, dx'
\end{multline}
Since $(u-v)^{\pm} \in H_0^1 (\Omega (\bold e) \cup \Gamma _2 (\bold e)) $ and $\gl _1 ^{\bold e}  \geq 0 $ it follows that
$(u-v)^{\pm} \equiv 0 $, i.e. 
$$
u \leq v= u_{\gs (\bold e)} \text{ in }  \gO (\bold e )
$$ 
\par
There are now two cases.\par
\smallskip
CASE 1 : $u-v \not \equiv 0 $ in $\gO (\bold e )$. \par
\smallskip
If this is the case then $u< v $ in $\gO (\bold e )$ by the strong maximum principle, and to conclude that there is a direction $\bold {e}' \in S^{N-2}$ such that $u \equiv u_{\gs (e')}$ and 
that  $u$ is sectionally foliated Schwarz symmetric we perform a
  \emph{rotating plane} procedure as in \cite{BaWi}, \cite{PaWe}, which is the analogous for rotations of the Alexandrov-Serrin Moving Plane Method (\cite{Se}, \cite{GiNiNi}) as generalized in  \cite{BeNi}.  \par
More precisely  if $\bold e = \bold e_{\gth _0}= (\cos (\gth _0 ), \sin (\gth _0),0 )$ is a direction and 
$u< u_{\gs (\bold e_{\gth _0} )} $ in $\gO (\bold e_{\gth _0} )$, we consider the set 
$$
\gTh = \{ \gth > \gth _0 : u<  u_{\gs (\bold e _{\gth '} )} \text{ in } \gO (\bold e_{\gth  '} ) \; \forall \gth ' \in (\gth _0, \gth)\} 
$$ 
\par
We show now that
the set $\gTh $ is nonempty and contains all the angles $\gth $ greater than and close to $ \gth _0 $.\par
In fact we can take a compact $K \subset \Omega $ such that the measure $|\Omega \setminus K |$ is small, and where 
$m= \min (u_{\gs (\bold e_{\gth _0} )} -u)>0 $. By continuity if $\gth $ is close to $ \gth _0 $ we have that
$u_{\gs (\bold e_{\gth } )} -u \geq \frac m2 >0 $. On the other hand on the boundary of the set $\Omega \setminus K $ we have that 
$u_{\gs (\bold e_{\gth } )} -u \geq 0 $, and since $f$ and $g$ are locally Lipschitz, it follows that
$ u_{\gs (\bold e_{\gth } )} -u $ satisfies a linear problem as \eqref{maximum}.\par
By the weak maximum principle \ref{massimodebole} we get that $u \leq u_{\gs (\bold e_{\gth } )} $ in $\gO (\bold e_{\gth } )$.\par
So the set $\gTh $ is nonempty and contains all the angles $\gth $ greater than and close to $ \gth _0 $.\par
Moreover it is bounded above by $\gth _0 + \gp $, since considering the opposite direction the inequality between $u$ and the reflected function get reversed. \par
  If $\gth _1= \sup \gTh $, then $u \equiv  u_{\gs (\bold e _{\gth _1}) } $ in $\gO (\bold e_{\gth _1} )$.
  For if this is not the case by the strong maximum principle (since by continuity 
 $u \leq  u_{\gs (\bold e _{\gth _1}) } $ in $\gO (\bold e_{\gth _1} )$ ) we would have 
  $u < u_{\gs (\bold e _{\gth _1}) }$ in $\gO (\bold e_{\gth _1})$, and using again the maximum principle in small domains and the previous technique we would get $u<  u_{\gs (\bold e _{\gth } )}$ in $\gO (\bold e_{\gth } )$ for $\gth > \gth _1= \sup \gTh $ and close to $\gth _1$.
 \par
 Note moreover that by construction for every direction $\bold e _{\gth } \in S^{1}$ either $u \leq u_{\gs (\bold e_{\gth } )} $ or
 $u \geq u_{\gs (\bold e_{\gth } )} $, so that $u$ is foliated Schwarz symmetric. \par
 \smallskip
 CASE 2: $u  \equiv u_{\gs (e)} $ in $\gO (\bold e )$. \par
 \smallskip
Let us consider
the usual cylindrical coordinates $(r, \gth,x_3)$ and the function $v=\frac {\partial u} {\partial \gth } $ (with any value when $r=0$, e.g. $v( 0)=0$). \par
Since $u \in C^1 (\overline {\Omega})$ and satisfies \eqref{genprob} with $f$ and $g$ satisfying \eqref{Regolarita-f,g}, by standard elliptic regularity (see e.g. \cite{GiTr}) we have that  $u \in C^2(\Omega) \cap C^1 (\overline {\Omega})$. 
Therefore if $\gf \in H^{1}_0 (\gO \cup \Gamma _2)$ and we  test \eqref{genprob} with the function
$ \frac {\partial \gf} {\partial \gth} $, we obtain easily that $v=\frac {\partial u} {\partial \gth } $ weakly satisfies  the problem
\be \label{autovalorezero}
 \begin{cases}
  - \gD v = \frac {\partial f}{\partial u}(|x|,u)  v   &  \text{ in }  \gO (\bold e) \\
   v =0  & \text{ on }( \Gamma _1 (\bold e)  ) \cup T(\bold e) \\
    \frac {\partial v}{\partial \gn} = \frac {\partial g}{\partial u}(|x|,u)  v& \text{ on } \Gamma _2 (\bold e)
 \end{cases}
\ee

\par
\smallskip
There are now  two possibility: either $\frac {\partial u} {\partial \gth }\equiv 0 $, i.e. $u$ is sectionally radial, or  
$
v=\frac {\partial u} {\partial \gth }\not \equiv 0 $.
In this latter case $v$ is an eigenfunction, with corresponding zero eigenvalue,  of the eigenvalue problem \eqref{AutovalCappe} in  $\gO (\bold e )$.
Since by construction $\gl _1 ^{\bold e}  \geq 0 $, we have that $v$
 is the first eigenfunction, with corresponding zero eigenvalue, so that $v=\frac {\partial u} {\partial \gth }$
 is strictly positive (or strictly negative) in $\gO (\bold e )$.
\par
% IMPORTANTE:  VERIFICARE REGOLARITA RICHIESTA A $u$ AFFINCHE ' $\frac {\partial u} {\partial \gth }$ SIA SOLUZIONE ! \par
This implies easily that the hypothesis in Lemma \ref{condsuffFSS} holds (see \cite{PaWe} for  the case of the  Dirichlet problem in a ball or annulus), so that $u$ is sectionally foliated Schwarz symmetric . \par
\medskip
In the general case, i.e. if $N>3$, having found a direction $\bold e \in S^{N-2}$,  as
in Lemma \ref{esistcappaconautovalpositivo} such that  $\gl _1 (\Omega (\bold e)) \geq 0 $,
it suffices to apply the previous procedure for every choice of cylindrical coordinates $(r, \gth, x'',x_N)$, with $x'' \in \R^{N-3}$, with respect to a couple of direction $\bold e $, $\bold e' \in S^{N-2}$ which determines the plane of the variables $(r, \gth)$. \par
This implies again that the hypothesis of Lemma \ref{condsuffFSS} is satisfied, so that $u$ is sectionally foliated Schwarz symmetric.
\end{proof}

\medskip

\begin{comment}

\subsection{Examples}
\par
DA FARE
\par
 \bigskip
[ \ \ \ Osservazione [ Neumann Problems ]
Il metodo dovrebbe funzionare anche per soluzioni di indice basso di problemi di Neumann, ad es nella palla intera (dove invece ovviamente il moving plane fallisce e le soluzioni non sono radiali) per dare la foliated schwarz symmetry, con massimo e minimo assoluti sul bordo (in met\`a dominio si prende Dirichlet nell' iperpiano centrale) ? \par
No, probabilmente il rotating plane non funziona e il risultato \`e vero solo per soluzioni di indice uno. \par
In ogni caso ammesso che sia vero, esistono ad esempio soluzioni di indice uno per Neumann ? \ \ \ ]

VEDERE REFERENZA COLORADO-PERAL

\end{comment}

\medskip

We end by remarking that exactly the same arguments in the previous proof show that
an analogous theorem holds for the Dirichlet problem in cylindrically symmetric domains:

\begin{theorem} \label{DirichletProblems}
 Let $\Omega $ be a bounded domain in $\R^N$, $N \geq 3$, which is cylindrically symmetric
 as in Definition \ref{DefDominiSimmetrici}.
 Let $u \in H^{1}_0 (\gO  ) \cap C^1 (\ov \Omega) $ be a weak solution of the problem
 \begin{equation} \label{Dirprob} 
\begin{cases} - \gD u =f(|x|,u) \quad &\text{in } \gO \\
u= 0 \quad & \text{on } \partial \Omega
\end{cases}
\end{equation}
 where 
$f$ has the form 
  $f(x,s)=f(|x'|,x_N,s)$,
  (i.e. $f(x',x_N,s)$  depends on $x'$ through the modulus $|x'|$)
  and $  f$ and $\frac {\partial f} {\partial s}$
are locally H\" older continuous functions in $ \overline {\Omega } \times \R $.
  \par
 Assume further that $f$ is strictly convex in the $s$- variable and that $u$ has Morse index $\gm (u) \leq N-1$.\par
 The $u$ is sectionally foliated Schwarz symmetric.
\end{theorem}


\begin{thebibliography}{99}
 
 \bibitem{AbDoOMe} E. Abreu, J. Marcos do O, E. Medeiros \emph{Properties of positive harmonic functions on the half space with a nonlinear boundary condition}, 
J. Diff.Eq. 248, 2010, 617-637

% \bibitem{AdYa} Adimurthi, S.L. Yadava \emph{Positive solution for Neumann Problem with critical nonlinearity on boundary},  

 
 \bibitem{Au} G. Auchmuty \emph{Steklov Eigenproblems and the Representation of Solutions of Elliptic Boundary Value Problems}, 
 Numerical Funct. Anal. Optim. 25, 3-4, 2004, pp. 321-348
 
  \bibitem{BaWi} T. Bartsch, T. Weth, M. Willem \emph{Partial symmetry of least energy nodal solutions to some variational problems}, J. Anal. Math 96, 2005, 1-18

 

\bibitem{BeFoSe} M. Ben Ayed, H. Fourti, A. Selmi, \emph{Harmonic Functions with nonlinear Neumann boundary condition and their Morse index}, 
Diff. Int. Eq., to appear
 
\bibitem{BeNi} H. Berestycki, L. Nirenberg, \emph{On the method of moving Planes and the Sliding Method}, Bol. Soc. Bras. Mat. 22, 1991, pp. 1-22

% \bibitem{BeNiVa} H. Berestycki, L. Nirenberg, S.R.S Varadhan, \emph{The principal eigenvalue and maximum principle for second order elliptic operators in general domains}, Comm.Pure 
% Appl. Math. 47 (1), 1994, pp. 47-92

% \bibitem{BePa} H. Berestycki, F. Pacella, \emph{Symmetry Properties for Positive Solutions of Elliptic Equations with Mixed Boundary Conditions},
 % J. Funct. Anal. 87 (1), 1989, pp. 177-211

\bibitem{Br} H. Brezis, \emph{Analyse Fonctionnelle}, Masson 1983

% \bibitem{CoPe} E. Colorado, I. Peral, \emph{Semilinear Elliptic Problems with Mixed Dirichlet-Neumann Boundary Conditions}, J. Funct. An. 199, 2003, pp.468-507

\bibitem{DaPa} L. Damascelli, F. Pacella, \emph{Symmetry results for cooperative elliptic systems via linearization}  
SIAM J. Math. Anal. 45 (2013), no. 3, 1003-1026

% \bibitem{DaGlPa1} L. Damascelli, F. Gladiali, F. Pacella, \emph{A symmetry result for semilinear cooperative elliptic systems}, in   Recent trends in nonlinear partial  
% differential equations. II. Stationary problems,  Amer. Math. Soc., Providence, RI,  Contemporary Mathematics 595 (2013), 187--204

\bibitem{DaGlPa} L. Damascelli, F. Gladiali, F. Pacella, \emph{Symmetry results for cooperative elliptic systems in unbounded domains}
  Indiana Univ. Math. J. 63 No. 3 (2014), 615--649

\bibitem{Ev} L. C. Evans,  \emph{Partial Differential Equations} - Graduate Studies in Mathematics 19, AMS, 1998.

\bibitem{GaRoSa} J. Garcia-Meli\`an, J.D. Rossi, J.C. Sabina de Lis \emph{Existence and Uniqueness of Positive Solutions to Elliptic Problems with Sublinear Mixed Boundary Conditions}, Comm. Contemp. Math 11, 2009, pp. 585-613

\bibitem{GiNiNi} B. Gidas, W.M. Ni, L. Nirenberg, \emph{Symmetry and related properties via the maximum principle}, Comm. Math. Phys. 68, 1979, pp. 209-243

\bibitem{GiTr} D.Gilbarg, N.S.Trudinger \emph{ Elliptic partial
differential equations of second order, $2^{ \text{nd} }$ edition}, 
1983, Springer 

\bibitem{GlPaWe}
F. Gladiali, F. Pacella, T. Weth, 
\emph{Symmetry and Nonexistence of low Morse index solutions in unbounded domains},
J. Math. Pures Appl. (9) 93 (5),  2010, pp. 536-558.

\bibitem{Ke} S. Kesavan, \emph{Topics in Functional Analysis and applications}, 1989 Wiley-Eastern. 

% \bibitem{KePa} S. Kesavan, F. Pacella, \emph{Symmetry of solutions of a system of semilinear elliptic equations}, Adv. Math. Sci. Appl. 9 n.1, 1999, pp. 361-369

\bibitem{LiPaTr}  P.L. Lions, F. Pacella, M. Tricarico, \emph{Best Constants in Sobolev Inequalities for Functions Vanishing on Some Part of the Boundary and Related Questions}, 
Indiana Univ. Math. J.  37 n.2, 1988, pp. 301-324
 
% \bibitem{MaVi} F. Maggi, C. Villani, \emph{Balls have the worst Best Sobolev Inequalities}, ??

\bibitem{Ma} N. Mavinga, \emph{Generalized eigenproblem and nonlinear elliptic equations with nonlinear boundary conditions}, Proc Royal Edinburgh, 142 A, 2012, pp. 137-153 
 
\bibitem{Pa}
F. Pacella,
\emph{Symmetry of Solutions to Semilinear Elliptic Equations with Convex Nonlinearities},
J. Funct. Anal. 192 (1), 2002, pp. 271-282

\bibitem{PaRa} F. Pacella, M. Ramaswamy, \emph{Symmetry of solutions of elliptic equations via maximum principle}, 
Handbook of Differential Equations: stationary partial differential equations.  Vol. VI  2008, 269-312

\bibitem{PaWe}
F. Pacella, T. Weth, 
\emph{Symmetry Results for Solutions of Semilinear Elliptic Equations via Morse index},
Proc. AMS 135 (6), 2007, pp 1753-1762

% \bibitem{PrWe}
%  M.H.Protter, H.F.Weinberger,
% \emph{ Maximum Principle in Differential Equations},
% Prentice-Hall, Inc., Englewood Cliffs, N.J., 1967.

\bibitem{Se} J. Serrin,  \emph{ A symmetry problem in potential theory},
    Arch. Ration. Mech. Anal.  43,  1971,  pp. 304--318.
    
     \bibitem{SmWi} D.Smets, M. Willem \emph{Partial symmetry and asimptotic behaviour for some elliptic variational problem}, Calc. Var. Part. Diff. Eq. 18, 2003, 57-75

    
 %   \bibitem{Te} S. Terracini,  \emph{ Symmetry properties of Positive Solutions to some Elliptic Equations with nonlinear Boundary Conditions problem in potential theory},
  %   Diff. Int. Eq.  8 (8),  1995,  pp. 1911-1922.






\end{thebibliography}
\end{document}